\documentclass[leqno,12pt]{amsart} %leqno is the option to put formula numbers on the left side
\usepackage{amssymb}
\theoremstyle{plain} %text of this environment is typesetted in italics
\newtheorem{theorem}{\indent\sc Theorem}[section]
\newtheorem{lemma}[theorem]{\indent\sc Lemma}

\newtheorem{claim}[theorem]{\indent\sc Claim}
\theoremstyle{definition} %text of this environment is typesetted in roman letters

\newtheorem{remark}[theorem]{\indent\sc Remark}

\def\Pi{{\mathbf{P}}}%  \Pi == \mathbf{P}
\begin{document}

\title[Spanning trees in a Claw-free graph]{Spanning trees in a Claw-free graph whose stems have at most $k$ branch vertices} %title of paper and the running head option

\author[P.~H.~Ha]{Pham Hoang Ha} %second author's name and the running head option

%\author[Y.~Kawakami]{Yu Kawakami} %first author's name and the running head option

%\dedicatory{Dedicated to Professor Miyuki Koiso on the occation of her sixtieth birthday}

%%%%%%%%%%%%%%% footnote %%%%%%%%%%%%%%%%
\keywords{spanning tree, stem, branch vertex, claw-free graph.}

\subjclass[2010]{ %2000 MSC numbers
Primary 05C05, 05C70. Secondary 05C07, 05C69.
}
%In case \subjclass[2000] command is not effective
%(or the version of amsart.cls is old), write as follows instead:

%%%%%%%%%%%% Authors addresses %%%%%%%%%%%%%
\address{% First Author
Department of Mathematics, Hanoi National University of Education, 136, XuanThuy str., Hanoi, Vietnam
}
\email{ha.ph@hnue.edu.vn}
%%%%%%%%%%%%%%%%%%%%%%%%%%%%%%%%%%%%%%%%%

\maketitle

\begin{abstract}
Let $T$ be a tree, a vertex of degree one and a vertex of degree at least
three is called a leaf and a branch vertex, respectively. The set of leaves of $T$
is denoted by $Leaf(T)$. The subtree $T-Leaf(T)$ of $T$ is called the stem of
$T$ and denoted by $Stem(T).$ In this paper, we give two sufficient conditions
for a connected claw-free graph to have a spanning tree whose stem has a bounded
number of branch vertices, and those conditions are best possible. As corollaries of main results we also give some conditions to show that a connected claw-free graph has a spanning tree whose stem is a spider. \\
\end{abstract}

\section{Introduction} 
 In this paper, we always consider simple graphs, which have neither loops nor multiple edges. For a
graph $G$, let $V(G)$ and $E(G)$ denote the set of vertices and the set of edges of G, respectively. We write $|G|$ for the order of $G$ (i.e., $|G| = |V(G)|).$ For a vertex $v$ of $G$, we denote by $\deg_{G}(v)$ the degree of $v$ in $G.$ For two vertices $u$ and $v$ of $G$,the distance between $u$ and $v$ in $G$ is denoted 
by $d_{G}(u, v)$.\\
For an integer $l\geqslant 2,$ let $\alpha^{l}(G)$ denote the number defined by
$$\alpha^{l}(G)=\max\{ |S|:S\subset V(G),d_{G}(x,y)\geqslant l\,\quad\text{for all distinct vertices }\,x,y\in S\}.$$
For an integer $k\geqslant$2, we define
$$\sigma_k^{l}(G)=\min\left\lbrace \sum_{a\in S}\deg_{G}(a):S\subset V(G),|S|=k, d_{G}(x,y)\geqslant l\quad\text{for all distinct vertices }\;x,y\in S\right\rbrace.$$
For convenience, we define $\sigma^{l}_{k}=+\infty$ if $\alpha^{l}(G)<k$. We note that, $\alpha^{2}(G)$ is often written $\alpha(G)$, which  is the \textit{independent number} of G, and $\sigma_k^{2}(G)$ is often witten $\sigma_{k}(G)$, which is the minimmum degree sum of $k$ independent vertices.

For a tree $T$, a vertex of degree at least three is called a \textit{branch vertex}, and a tree having at most one branch vertex is called a \textit{spider}. Many researchers have investigated the independent number conditions and the degree sum conditions for the existence of a spanning tree with bounded number of branch vertices or it is a spider (see \cite{GH}, \cite{GHHSV}, \cite{AK} and \cite{OY} for examples). A vertex of $T$, which has degree one, is often called a \textit{leaf} of $T$, and the set of leaves of $T$ is denoted by $Leaf(T)$. Many results were studied on the independent number conditions and the degree sum conditions for the existence of a spanning tree with bounded number of leaves (also see \cite{AK} and \cite{OY} for examples). Moreover, many analogue results for the claw-free graph are studied (see \cite{KKMO} and \cite{MOY} for examples). 

The subtree $T-Leaf(T)$ of $T$ is called the $stem$ of $T$ and is denoted by $Stem(T)$. Recently, M. Kano and his collaborations gave an innovation by studying a spanning tree in a graph with specified stem. We introduce here some of them. Their first result is the following.

\begin{theorem}[{\cite[Kano, Tsugaki and Yan]{KTY}}]\label{thm1}
Let $k\geqslant 2$ be an integer, and $G$ be a connected graph. If $\sigma_{k+1}(G)\geqslant|G|-k-1$, then $G$ has a spanning tree whose stem has maximum degree at most $k$.
\end{theorem}
After that, the sufficient conditions for a connected graph to have a spanning tree whose stem has a few number of leaves were introduced as the following theorems.
\begin{theorem}[{\cite[Tsugaki and Zhang]{TZ}}]
Let $G$ be a connected graph and $k\geqslant 2$ be an integer. If $\sigma_{3}(G)\geqslant|G|-2k+1$, then $G$ have a spanning tree whose stem has at most $k$ leaves.
\end{theorem}
\begin{theorem}[{\cite[Kano and Yan]{KY}}]
Let $G$ be a connected graph and $k\geqslant 2$ be an integer. If $\sigma_{k+1}(G)\geqslant|G|-k-1$, then G have a spanning tree whose stem has at most $k$ leaves.
\end{theorem}
\begin{theorem}[{\cite[Kano and Yan]{KY}}]
	Let $G$ be a connected claw-free graph and $k\geqslant 2$ be an integer. If $\sigma_{k+1}(G)\geqslant|G|-2k-1$, then G have a spanning tree whose stem has at most $k$ leaves.
\end{theorem}
The following theorem gives two sufficient conditions for a connected graph to have a spanning tree whose stem has a bounded number of branch vertices.
\begin{theorem}[{\cite[Yan]{Yan}}] \label{thm-1}
Let $G$ be a connected graph and $k$ be a non-negative integer. If one of the following conditions holds, then $G$ have a spanning tree whose stem has at most $k$ branch vertices.
\begin{enumerate}
\item[(i)]$\alpha^{4}(G)\leqslant k+2.$
\item[(i\hspace{-.1em}i)]$\sigma^{4}_{k+3}(G)\geqslant |G|-2k-3.$
\end{enumerate}
\end{theorem}
When $k=1,$ Theorem \ref{thm-1} gives a previous result as the following.
\begin{theorem}[{\cite[Kano and Yan]{KY15}}]
	Let $G$ be a connected graph. If $\sigma^{4}_{4}(G)\geqslant |G|-5,$ then $G$ have a spanning tree whose stem is a spider.
\end{theorem}
\begin{remark}
	We remark that all conditions which were mentioned above are best possible.
\end{remark}
We are very interested in this topic. So, we would like to study the spanning tree in a graph with specified stem. The purpose of this paper is to give some sufficient conditions for a connected claw-free graph to have a spanning tree whose stem has at most $k$ branch vertices. In particular, our main theorems are the followings.
\begin{theorem}\label{thm-main}
	Let $G$ be a connected claw-free graph and $k$ be a non-negative integer. If $\sigma^{4}_{k+3}(G)\geqslant |G|-2k-5$, then $G$ has a spanning tree whose stem has at most $k$ branch vertices.
	\end{theorem}
\begin{theorem}\label{thm-main2}
	Let $G$ be a connected claw-free graph and $k$ be a non-negative integer. If $\sigma^{5}_{2}(G)\geqslant |G|-3k-6$, then $G$ has a spanning tree whose stem has at most $k$ branch vertices.
\end{theorem}
Applying the main theorems with $k=1,$ we give the following.
\begin{theorem}\label{thm-main1}
	Let $G$ be a connected claw-free graph. If $\sigma^{4}_{4}(G)\geqslant |G|-7$ or $\sigma^{5}_{2}(G)\geqslant |G|-9$, then $G$ has a spanning tree whose stem is a spider.
\end{theorem}
\section{Sharpness of Theorem \ref{thm-main} and Theorem \ref{thm-main2}}
We first show that the condition of Theorem \ref{thm-main}
is best possible. Let  $m,k\geqslant 1$ be integers, and let $D_{1},\dots,D_{k+3}$ be disjoint copies of $K_{m}$ and $D=K_{k+3}$ with distinct vertices $z_1, ..., z_{k+3}$. Let $v_{1},\dots,v_{k+3}$ be vertices not contained in $D\cup D_{1}\cup\dots\cup D_{k+3}$. Join $z_{i},v_{i}$ to all vertices of $D_{i}(1\leqslant i\leqslant k+3)$ by edges, respectively. Let $G$ denote the resulting graph. Then $G$ is a connected claw-free graph. Setting $V= \{v_1, \dots , v_{k+3}\}.$ We are easy to see that for each set $S$ such that $S\subset V(G),|S|=k+3$ and $ d_{G}(x,y)\geqslant 4$ for all distinct vertices $x,y\in S$, then $S=V$ or $S=(V\setminus\{v_j\})\cup \{y_j\},$ where $y_j \in D_j.$ Then we can compute that $\sum_{a\in S}\deg_{G}(a) = |G|-2k-6$ for the first case and $\sum_{a\in S}\deg_{G}(a) = |G|-2k-5$ for the last case. So we get $\sigma^{4}_{k+3}(G)=|G|-2k-6$. On the other hand, since for any spanning tree $T$ of $G,$ then there are at least $k+1$ points in the set $\{z_{1},z_{2},\dots,z_{k+3}\}$ must be the branch vertices of $Stem(T).$ So $G$ has no spanning tree who stem has at most $k$ branch vertices. Therefore, the condition of Theorem \ref{thm-main} is best possible.

Now we also consider graph $G$ above with $m=1.$ So we may see that $\sigma^{5}_{2}(G)=2= |G|-3k-7.$ Moreover, for every spanning tree $T$ of $G,$ then there are at least $k+1$ points in the set $\{z_{1},z_{2},\dots,z_{k+3}\}$ must be the branch vertices of $Stem(T).$ This shows that the condition of Theorem \ref{thm-main2} is best possible.
\section{Proofs of Theorem \ref{thm-main} and Theorem \ref{thm-main2}}
Beside of giving some refinements of the proofs in \cite{KY} and \cite{Yan}, we will use some analogue arguments of them to prove Theorem \ref{thm-main} and Theorem \ref{thm-main2}.\\
Firstly, we recall the following useful lemma.
\begin{lemma}\label{main-lem1}
Let $T$ be a tree, and let $X$ be the set of vertices of degree at least 3. Then the number of leaves in $T$ is counted as follow:
$$|Leaf(T)|=\sum_{x\in X}(\deg_{T}(x)-2)+2.$$
\end{lemma}
Assume that $G$ satisfies the condition in Theorem \ref{thm-main} and does not have spanning tree whose stem has at most $k$ branch vertices. We choose a tree $T$ whose stem has $k$ branch vertices in $G$ so that
\begin{enumerate}
\item[(C1)]$|T|$ is as large as possible.
\item[(C2)]$|Leaf(Stem(T))|$ is as small as possible subject to (C1).
\item[(C3)]$|Stem(T)|$ is as small as possible subject to (C1), (C2).
\end{enumerate}
By the choice (C1), we have the following claim.
\begin{claim}\label{claim 1}
For every $v\in V(G)-V(T), N_{G}(v)\subseteq Leaf(T)\cup(V(G)-V(T))$.
\end{claim}
$Stem(T)$ has $k$ branch vertices. Denote the number of leaves of $Stem(T)$ by $l$. By Lemma \ref{main-lem1}, $Leaf(Stem(T))=l\geqslant k+2$. Let $x_{1},x_{2},\dots,x_{l}$ be the leaves of $Stem(T)$. Since $T$ is not a spanning tree of $G$, there exist two vertices $v\in V(G)-V(T)$ and $u\in Leaf(T)$ which are adjacent in $G$. Thus, we have the following claim.
\begin{claim}\label{claim 1b}
If $u$ is adjacent with a vertex $w$ of $stem(T)$ then $\deg_{stem(T)}(w) = 2.$
\end{claim}
\begin{proof}
	Suppose that $u$ is adjacent with a vertex $w$ but $\deg_{stem(T)}(w) \not= 2.$\\
	If $\deg_{stem(T)}(w) = 1$ then $z$ is a leaf of $stem(T).$ We consider a new tree $T_1 = T +uv$ then $stem(T_1)$ have $k$ branch vertices and $|T_1| > |T|.$ This contradicts the condition (C1).\\
	If $\deg_{stem(T)}(w) \geq 3$ than $w$ is a branch vertex of $stem(T).$ We also consider a new tree $T_1 = T +uv$ then $stem(T_1)$ also have $k$ branch vertices and $|T_1| > |T|.$ This contradicts the condition (C1). Claim \ref{claim 1b} is proved.
	\end{proof}
Now, we use the properties of the claw-free graph to give the following claim.
\begin{claim}\label{claim 6}
	Set $M = \{ w \in Stem(T)|\deg_{Stem(T)}(w)=2 \}.$ Then $|M| \geq 3$.
\end{claim}
\begin{proof}
	Otherwise, by Claim \ref{claim 1b} we have two following cases.
	
	Case 1. $|M| = 1.$ We call $w \in M$ then $u$ is adjacent with $w$ by Claim \ref{claim 1b}. Let $y, t$ be two adjacent vertices of $w$ in $Stem(T).$ Here $y$ and $t$ are branch vertices, leaves or one leaf and one branch vertex.  By definition of the claw-free graph, then either $uy$ or $ut$ or $yt$ is an edge in $G.$ We consider a new tree 
	\begin{align*}
		T_2 = 
		\begin{cases}
			T-wu +uy+uv & \text { if $uy \in E(G)$},\\
			T-tw +ut +uv& \text { if $ut \in E(G)$},\\
			T-tw +ty+uv & \text { if $yt \in E(G)$}.
		\end{cases}
	\end{align*}
	Then, by $y$ is a branch vertex or a leaf of $T$, the resulting tree $T_2$ of $G$ is a tree whose stem has $k$ branch vertices and the order of the resulting tree is greater than $|T|$, which contradicts the condition (C1).
		
	Case 2. $|M| = 2.$ We call $w_1, w_2 \in M.$ Without loss of generality we may assume that $u$ is adjacent with $w_1$ by Claim \ref{claim 1b}. If $w_1$ is not adjacent with $w_2$ in $Stem(T)$ then by using the same arguments in case 1 we get a contradiction. On the other hand, if $w_1$ is adjacent with $w_2$ then let $y$ be another adjacent vertex of $w_1$ in $Stem(T).$ By definition of the claw-free graph, then either $uy$ or $uw_2$ or $yw_2$ is an edge in $G.$ We consider a new tree 
	\begin{align*}
		T_2 = 
		\begin{cases}
			T-yw_1 +uy+uv & \text { if $uy \in E(G)$},\\
			T-w_1w_2 +uw_2 +uv& \text { if $uw_2 \in E(G)$},\\
			T-w_1w_2 +yw_2+uv & \text { if $yw_2 \in E(G)$}.
		\end{cases}
	\end{align*}
	Then, by $y$ is a branch vertex or a leaf of $T$, resulting tree $T_2$ of $G$ is a tree whose stem has $k$ branch vertices and the order of the resulting tree is greater than $|T|$, which contradicts the condition (C1).\\
Claim \ref{claim 6} is proved.	
\end{proof}
\begin{claim}\label{claim 2}
Leaf(Stem(T)) is an independent set of $G.$
\end{claim}
Assume that there exists two vertices $x_{i}$ and $x_{j}$ of $Leaf (Stem(T ))$ which are adjacent in $G.$ Then add $x_{i}$ and $x_{j}$ to $T$. The resulting subgraph of $G$ includes a unique
cycle, which contains an edge $e_{1}$
of $Stem(T)$ incident with a branch vertex. By removing the edge $e_{1}$, we obtain the resulting tree $T_3$ such that $Stem(T_3)$ has at most $k$ branch vertices, $|T_3|=|T|$ and $|Leaf(Stem(T_3))\leqslant|Leaf(Stem(T))|-1$. If $Stem(T)$ 
has $k- 1$ branch vertices, then add $uv$ to $T_3$
; we obtain a tree whose stem has at
most $k$ branch vertices and the order of the tree is greater than $|T|$, which contradicts the condition (C1). Otherwise, $T_3$
contradicts the condition (C2). Hence
$Leaf(Stem(T))$ is an independent set of $G$.
\begin{claim}\label{claim 3}
	For every 
	$x_{i}(1 \leqslant i \leqslant l)$, there exists a vertex $y_{i}\in Leaf(T)$ adjacent to $x_{i}$
	and $N_{G}(y_{i})\subset Leaf(T)\cup\left\{x_{i}\right\}$.
\end{claim}
\begin{proof}
	It is easy to see that for each leaf $x\in Leaf(Stem(T)),$ there exists at least a vertex $y$ in $Leaf(T)$ adjacent to $x.$ Now, for every leaf $y$ of $T$ adjacent to a leaf of $Stem(T)$ in $T$, $y$ is not adjacent to any vertex of $V(G)-V(T).$ Indeed, otherwise we can add an edge joining $y$ to a vertex of $V(G)-V(T)$ to $T$ then the resulting tree contradicts the condition (C1).
	\\\indent Suppose that for some $1\leqslant i\leqslant l$, each leave $y_{i_j}$ of $T$ adjacent to $x_{i},$ is also adjacent to a vertex $z_{i_j}\in (Stem(T)-\left\{x_{i}\right\})$. Then for every leaf $y_{i_j}$ adjacent to $x_{i}$ in $T$, remove the edge $y_{i_j}x_{i}$ from $T$ and add the edge $y_{i_{j}}z_{i_{j}}$. Denote the resulting tree of $G$ by $T_{4}$. Then $T_{4}$ is a tree which has at most $k$ branch vertices. If $x_{i}$ is adjacent with a branch of $Stem(T)$, then $Leaf(Stem(T_{4}))=Leaf(Stem(T))-\{x_{i}\}$, which contradicts the condition (C2). If $x_{i}$ is not adjacent with a branch of $Stem(T)$, then $Stem(T_{4})=Stem(T)-\left\{x_{i}\right\}$, which contradicts the condition (C3). Therefore, the claim holds.
\end{proof}
\begin{claim}\label{claim 4}
For any two dictinct vertices $y,z\in \left\{v,y_{1},y_{2},\dots,y_{l}\right\},d_{G}\left(y,z\right)\geqslant 4$.
\end{claim}
\begin{proof}
First, we show that $d_{G}\left(v,y_{i}\right)\geqslant 4$ for every $1\leqslant i\leqslant l$. Let $P_{i}$ be the shortest path connecting $v$ and $y_{i}$ in $G$.  If all the vertices of $P_{i}$ between $v$ and $y_{i}$ are contained in $Leaf(T)\cup (V(G)-V(T))\cup\left\{x_{i}\right\}$. Then add $P_{i}$ to $T$ (if $P_{i}$ passes through $x_{i}$, we just add the segment of $P_{i}$ between $v$ and $x_{i})$ and remove the edges of $T$ joining $V(P_{i}\cap Leaf(T))$ to $V(Stem(T))$ except the edge $y_{i}x_{i}.$ Then resulting tree of $G$ is a tree whose stem has at most $k$ branch vertices and the order of the resulting tree is greater than $|T|$, which contradicts the condition (C1). Then there exists a vertex $s\in V(P_{i})$ with $s\in V(Stem(T))-\left\{x_{i}\right\}$. Hence, by Claim \ref{claim 1} and Claim \ref{claim 3}, $d_{G}(v,s)\geqslant 2$ and $d_{G}(s,y_{i})\geqslant 2$. Therefore, $d_{G}(v,y_{i})=d_{G}(v,s)+d_{G}(s,y_{i})\geqslant 4.$\\
\indent Next, we show that $d_{G}(y_{i},y_{j})\geqslant 4$ for all $1\leqslant i<j\leqslant l$. Let $P_{ij}$ be the shortest path connecting $y_{i}$ and $y_{j}$ in $G$. We note that if $P_{ij}$ passes through $x_{i}$ (or $x_{j}$), then $y_{i}x_{i}\in E(P_{ij})$ (or $y_{j}x_{j}\in E(P_{ij})$), respectively. If all vertices of $P_{ij}$ between $y_{i}$ and $y_{j}$ are contained in $Leaf(T)\cup(V(G)-V(T))\cup\left\{x_{i},x_{j}\right\}$. Then add $P_{ij}$ to $T$ to remove the edges of $T$ joining $V(P_{ij }\cap Leaf(T))$ to $V(Stem(T))$ except the edges $y_{i}x_{i}$ and $y_{j}x_{j}$. Then the resulting graph of $G$ includes a unique circle, which contains an edge $e_{2}$ of $Stem(T)$ incident with a branch vertex. By removing the edge $e_{2}$, we obtain a tree $T_{5}$ whose stem has at most $k$ branch vertices. If $P_{ij}$ contains a vertex of $V(G)-V(T)$, then the order of $T_{5}$ is greater than $T$, which contradicts the condition (C1). Otherwise, $|T_{5}|=|T|$ and $|Leaf(Stem(T_{5}))|=|Leaf(Stem(T))|-1$. This contradicts the condition (C2). Hence, $P_{ij}$ passes through a vertex $s\in Stem(T)-\left\{x_{i},x_{j}\right\}$. Then there exists a vertex $ t\in V(P_{ij})$ with $t\in V(Stem(T))-\left\{x_{i},x_{j}\right\}$. Hence, by Claim \ref{claim 3}, $d_{G}(y_{i},s)\geqslant 2$ and $d_{G}(s,y_{j})\geqslant 2$. Therefore, $d_{G}(y_{i},y_{j})=d_{G}(y_{i},s)+d_{G}(s,y_{j})\geqslant 4$ for $1\leqslant i<j\leqslant k$. 
\end{proof}
As a corollary of Claim \ref{claim 4}, we have the following claim.
\begin{claim}\label{claim 5}
 $N_{G}(v)\cap N_{G}(y_{i})=\emptyset$ and $N_{G}(y_{i})\cap N_{G}(y_{j})=\emptyset$ for $1\leqslant i\neq j\leqslant l$.
\end{claim}
Denote $Y=\left\{y_{1},y_2,\dots,y_{l}\right\}$. Since Claim \ref{claim 1}-\ref{claim 5}, we have
\begin{center}
$N_{G}(v)\subseteq (V(G)-V(T)-\left\{v\right\})\cup (N_{G}(v)\cap(Leaf(T)-Y)),$\\
$\displaystyle\bigcup_{i=1}^{k+2}N_{G}(y_{i})\subseteq (Leaf(T)-Y-N_{G}(v))\cup\left\{x_{1},\dots,x_{k+2}\right\}$
\end{center}
Using Claim \ref{claim 6}, we have $|Stem(T)| \geq l + k + 3.$\\
Hence by setting $h=|N_{G}(v)\cap(Leaf(T)-Y)|$, we have
\begin{equation*}
\begin{split}
\deg_{G}(v)+\displaystyle\sum\limits_{i=1}^{k+2}\deg_{G}(y_{i})&\leqslant|G|-|T|-1+h+|Leaf(T)|-h-l+k+2\\
&=|G|-|Stem(T)|-l+k+1\\
&\leqslant |G|-2l-2\leqslant|G|-2k-6 \text{ (by $l \geq k+2$)}.
\end{split}
\end{equation*}
Which contradicts the condition in Theorem \ref{thm-main}.\\
Theorem \ref{thm-main} is proved.

Now, since the properties of the claw-free graph we have the following claim.
\begin{claim}\label{claim 4b}
	$d_{G}\left(v,y_i\right)\geqslant 5$ for all $1\leq i \leq l.$
\end{claim}
\begin{proof}
Since Claim \ref{claim 4}, we have $d_{G}(v,y_{i})\geqslant 4.$ Assume that $d_{G}(v,y_{i})= 4.$ Let $P_{i}$ be the shortest path connecting $v$ and $y_{i}$ in $G$. Then by the proof of Claim \ref{claim 4} there exists a vertex $s\in V(P_{i})$ with $s\in V(Stem(T))-\left\{x_{i}\right\}$. By Claim \ref{claim 1} and \ref{claim 3}, we have $d_{G}(v,s)\geqslant 2$ and $d_{G}(s,y_{i})\geqslant 2.$ Therefore, if $d_{G}(v,y_{i})=d_{G}(v,s)+d_{G}(s,y_{i})= 4$ then $d_{G}(v,s)=d_{G}(s,y_{i})= 2$ and, moreover, $s$ must be in $M$ by the proof of Claim \ref{claim 1b}. Let $vu's$ and $y_iz_is$ be two paths in $P_i.$ So $u' \in Leaf(T)\cup (V(G)-V(T)).$\\	
	If $u' \in V(G) - V(T)$ then we consider a tree $T_*= T+su'$ then $Stem(T_*)$ has $k$ branch vertices and $|T_*| > |T|.$ This contradicts (C1). \\	
	If $u' \in Leaf(T),$ remove the adge of $T$ joining $u'$ and add $u's$ to $T.$ Then resulting tree $T_*$ of $G$ whose stem has $k$ branch vertices $,|T_*| =|T|, |Leaf(Stem(T_*))|=|Leaf(Stem(T))|$ and $|Stem(T_*)|=|Stem(T)|.$ Using the same arguments in the proofs of Claim \ref{claim 1b} and \ref{claim 6} we can show that $\deg_{stem(T_*)}(s)= 2$ and if $s$ is adjacent with two vertices $y, t$ in $Stem(T_*)$ then $\deg_{stem(T_*)}(y)= \deg_{stem(T_*)}(t)= 2.$ Now, by definition of the claw-free graph $G$, then either $u't$ or $u'y$ or $ty$ is an edge in $G.$ Let $p$ be a vertex of $Stem(T)$ such that $z_ip$ is in $E(T).$  We consider a new tree 
	\begin{align*}
		T_6 = 
		\begin{cases}
			T_*-ys +u'y+u'v & \text { if $u'y \in E(G)$},\\
			T_*-ts +u't +u'v& \text { if $u't \in E(G)$},\\
			T_*-ts-ys-z_ip +ty+sz_i+z_iy_i+su'+u'v & \text { if $ty \in E(G)$,}
		\end{cases}
	\end{align*}
	Then  resulting tree $T_6$ of $G$ is a tree whose stem has $k$ branch vertices and the order of $T_6$ is greater than $|T_*|=|T|$, which contradicts the condition (C1). So $d_{G}(v,y_{i})\geq 5.$
	\end{proof}
Fix an index $i$. Since Claim \ref{claim 1}-\ref{claim 4b}, we have
\begin{center}
	$N_{G}(v)\subseteq (V(G)-V(T)-\left\{v\right\})\cup (N_{G}(v)\cap(Leaf(T)-Y)),$\\
	$N_{G}(y_{i})\subseteq (Leaf(T)-Y-N_{G}(v))\cup\{x_{i}\}$
\end{center}
Using Claim \ref{claim 6}, we have $|Stem(T)| \geq l + k + 3.$\\
Hence, we have
\begin{equation*}
	\begin{split}
		\deg_{G}(v)+\deg_{G}(y_{i})&\leqslant|G|-|T|-1+|Leaf(T)|-l+1\\
		&=|G|-|Stem(T)|-l\\
		&\leqslant |G|-2l-k-3\leqslant|G|-3k-7 \text{ (by $l \geq k+2$)}.
	\end{split}
\end{equation*}
Which contradicts the condition in Theorem \ref{thm-main2}.\\
Theorem \ref{thm-main2} is proved.

\end{document}